%% file: v3.tex
\documentclass[10pt]{amsart}
\input{amssymbol}

\usepackage{enumitem}
\usepackage[utf8]{inputenc}
\usepackage{amsmath}
\usepackage{amssymb}
\usepackage{amsthm}
\usepackage{amsmath,amscd}
\usepackage{xcolor}
\usepackage{parskip}
\usepackage[colorlinks=true, linkcolor=blue]{hyperref}
\usepackage{tikz-cd}
\title
{On the Frobenius fields of abelian varieties over number fields}
\author{Ashay A. Burungale}
\address{ Department of Mathematics~\\UT Austin~\\
Austin, TX 78712}
\email{ashayburungale@gmail.com}

\author{Haruzo Hida}
\address{
Department of Mathematics~\\UCLA ~ \\
Los Angeles, CA 90095-1555, USA
}
\email{hida@math.ucla.edu}

\author{Shilin Lai}
\address{Department of Mathematics~\\UT Austin~\\
Austin, TX 78712}
\email{shilin.lai@math.utexas.edu}

\date{\today}

\newcommand{\G}{\mathbf{G}}
\newcommand{\bT}{\mathbf{T}}

\newcommand{\reg}{\mathrm{reg}}
\newcommand{\mb}[1]{\mathbf{#1}}

\newcommand{\Frob}{\mathrm{Frob}}

\setlist[enumerate]{label=(\arabic*).}

\begin{document}

\begin{abstract}
Let $A$ be a non-CM simple abelian variety over a number field $K$. 
For a place $v$ of $K$ such that $A$ has good reduction at $v$, let $F(A,v)$ denote the Frobenius field generated by the 
corresponding Frobenius eigenvalues.
If $A$ has connected monodromy groups, we show that the set of places $v$ such that $F(A,v)$ is isomorphic to a fixed number field has upper Dirichlet density zero. Moreover, assuming the GRH, we give a power saving upper bound for the number of such places.
\end{abstract}

\maketitle
\setcounter{tocdepth}{1}
\tableofcontents

\section{Introduction}\label{s:Int}

For an abelian variety over a number field and a place of good reduction, a basic invariant is the Frobenius field generated by the corresponding Frobenius eigenvalues. In this paper we study its connection with the arithmetic of the abelian variety. 

For a CM abelian variety, in view of the CM theory of Shimura--Taniyama--Weil, the Frobenius fields are contained in a fixed number field and equal to it for a set of places of Dirichlet density one. A natural question: to explore the upper Dirichlet density of the set of places at which the Frobenius field of a non-CM abelian variety coincides with a given number field up to an isomorphism. The question has been studied in the literature in various low-dimensional cases, and the primary goal of this paper is to consider the general case via a uniform approach.

We show that the density is zero under a mild connected-ness hypothesis (see Theorem~\ref{thmA}). Moreover, assuming the GRH, we provide a power saving upper bound on the size of the set of places with bounded norm 
at which the Frobenius field of a non-CM abelian variety coincides with a given number field (see Theorem~\ref{thmB}). 

%which also naturally incorporates refinements previously used in the elliptic curves case.
%In the setting of prior works, such as elliptic curves, our method yields a better exponent than the one obtained using Sieve method.  

\subsection*{Main results}
Let $K$ be a number field and $\Sg_{K}$ the set of its finite places. 
For $\overline{K}$ an algebraic closure, let 
$G_{K}=\Gal(\overline{K}/K)$ be the corresponding absolute Galois group. 
For a place $v \in \Sg_{K}$, let ${\rm Frob}_v$ be an associated geometric Frobenius.

Let $A$ be an abelian variety defined over $K$ of dimension $g  $ and conductor $\mathfrak{N}$. 
For a prime $p$, let $T_{p}A$ be the $p$-adic Tate module of $A$ and 
$$\rho_{A,p}: G_{K} \rightarrow \Aut_{\Z_{p}} T_{p}A$$
 the associated $p$-adic Galois representation. 
For $v\ndivide p\mathfrak{N}$, it is unramified at $v$. 
Let $F(A,v)$ be the splitting field of the characteristic polynomial of $\rho_{A,p}({\rm Frob}_{v})$. 
As $(\rho_{A,p})_{p}$ is a compatible system of Galois representations, $F(A,v)$ is a number field independent of $p$, referred to as the Frobenius field associated to the pair $(A,v)$.

Let $M$ be a number field and $S_{A,M}$ a subset of finite places $\Sg_{K}$ of $K$ given by 
$$
S_{A,M}=\{v \in \Sg_{K}\,|\,v \ndivide \mathfrak{N}, F(A,v)\cong M\}. 
$$

When $A$ is a CM abelian variety, the Frobenius fields are a subfield\footnote{More precisely, the Frobenius fields are a subfield of the CM endomorphism field $K$ given by  $\alpha^{\Phi'}\in K$  for some  $\alpha\in K'$  for the reflex CM type $(K',\Phi')$.} of the corresponding CM field (cf.~\cite{Sh}). Moreover, for a Dirichlet density one subset, the Frobenius fields equal a fixed subfield of 
the CM field.
%up to isomorphism.

In this paper, for non-CM abelian varieties $A$, we consider dependence of the Frobenius fields $F(A,v)$ on the place $v$. 
To state the results, we recall the following notion. For $S \subset \Sg_{K}$, 
the upper Dirichlet density ${\rm ud}(S)$ is given by 
$$
{\rm ud}(S) = \limsup_{X \rightarrow \infty}\frac{\#\{v \in \Sg_{K}: \mathbf{N}_{K/\Q}v \leq X, v \in S\}}{\#\{v \in \Sg_{K}: \mathbf{N}_{K/\Q}v \leq X\}}, 
$$
where $\mathbf{N}_{K/\Q}$ denotes the norm of the extension $K/\Q$.

Our first main result is the following.

\begin{thm}\label{thmA}
Let $A$ be an absolutely simple non-CM abelian variety defined over a number field $K$. Suppose that the monodromy groups associated to $A$ over $K$ are connected (cf.~Hypothesis~\ref{hyp:mc}). 
Then for any number field $M$, we have 
$$
{\rm ud}(S_{A,M})=0.
$$
\end{thm}

The connectivity Hypothesis~\ref{hyp:mc} is satisfied by any abelian variety $A$ over some finite extension of its field of definition (cf.~Remark~\ref{rmk:Se}). Moreover, it is satisfied by a class of abelian varieties over their field of definition, 
for instance: most abelian varieties in a typical family (e.g. hyperelliptic Jacobians \cite[Theorem 1.2]{ZarhinJacobian}).

\begin{remark}
Theorem \ref{thmA} gives a criterion for characterisation of CM/non-CM abelian varieties. For different criteria in the case of $\GL_2$-type abelian varieties over $\Q$, the reader may refer to \cite[\S3.1.1]{Hi5}.  
\end{remark}

In view of Theorem~\ref{thmA} it is natural to seek to estimate the size of subsets of $S_{A,M}$ with bounded norm. To state the result, we introduce some notation. Let $\G$ be the Mumford--Tate group associated to $A$. Let $\G^\mathrm{ss}$ be its semisimple quotient. Let
\[
  d=\dim\G^\mathrm{ss},\quad r=\rank\G^\mathrm{ss}
\]
denote the dimension and (absolute) rank of $\G^\mathrm{ss}$ respectively. With these notations, our quantitative result is as follows.

\begin{thm}\label{thmB}
Let $A$ be an absolutely simple non-CM abelian variety defined over a number field $K$. Suppose that the monodromy groups associated to $A$ over $K$ are connected (cf.~Hypothesis~\ref{hyp:mc}), and that the generalized Riemann hypothesis (GRH) holds for $L$-functions as in Remark~\ref{rmk:GRH}.
Then for any $\varepsilon>0$, there exists a constant $c$ depending on $(A,K,\varepsilon)$ such that 
 for any  number field $M$ and $X\in \R_{>0}$, we have
\[
        |\{v\in S_{A,M}\,|\,\mathbf{N}_{K/\Q}v\leq X\}|\leq cX^{1-\frac{1}{3d-r+2}+\varepsilon}.
    \]
   
\end{thm}

%In particular, 
If $A$ is generic\footnote{That is, the associated adelic Galois representation has open image in ${\rm GSp}_{2g}(\wh{\Z})$.}, then $\G^\mathrm{ss}=\mathrm{PGSp}_{2g}$ for an integer $g\geq 1$. In this case $d=2g^2+g$ and $r=g$, so we obtain the following.
\begin{cor}
In the setting of Theorem~\ref{thmB}  suppose further that $A$ is generic and let $\G^\mathrm{ss}=\mathrm{PGSp}_{2g}$ for an integer $g\geq 1$. Then for any $\varepsilon>0$, there exists a constant $c$ depending on $(A,K,\varepsilon)$ such that
for any  number field $M$ and $X\in \R_{>0}$, we have
\[
    |\{v\in S_{A,M}\,|\,\mathbf{N}_{K/\Q}v\leq X\}|\leq c X^{1-\frac{1}{6g^2+2g+2}+\varepsilon}.
  \]
\end{cor}

\begin{remark}\label{rmk:GRH}
    Our approach to Theorem~\ref{thmB} is conditional  on the GRH for two types of $L$-functions. The first is the Dedekind zeta function for a number field $M$. The second consists of $L$-functions of finite order Hecke characters
    %, more precisely those characters 
    appearing in the statement of Lemma~\ref{lem:Supersolvable}, namely the  characters attached to abelian subquotients of Galois extensions of the form $K(A[d])/K$, where $d$ is a square-free product of rational primes.
\end{remark}

\subsection*{About the proof}
Our approach is based  on the compatible system of Galois representations associated to the abelian variety and some  group theory. 
To begin, we recast the problem in terms of (a variant of) Frobenius tori in the algebraic monodromy group of the abelian variety. 
A volume computation of conjugacy classes arising from these tori and large Galois image results for $A$ are keys of the proof. 

We now describe the strategy in more detail. 
To begin, for primes $p$ which split completely in $M$, we show that if $v\in S_{A,M}$, then the image of $\Frob_v$ in the associated mod $p$ monodromy group lies in a $\F_p$-rational Borel subgroup. Then we give a volume upper bound for the union of all such subgroups (cf.~Corollary~\ref{cor:UpperBound}), perhaps of independent interest. On the other hand, since $A$ is non-CM, the associated Galois representations have large image, thanks to the work of Wintenberger \cite{W} and Hui--Larsen \cite{HL0} (cf.~Theorem~\ref{thm:HL}). In light of the image lower bound and the volume upper bound, the Chebotarev density theorem implies that $\mathrm{ud}(S_{A,M})$ is bounded above by a constant $c$ less than $1$, which is independent of the prime $p$. This deduction relies on basic properties of reductive groups, several of which require the group to be connected.
 
Next, we synthesise the analysis at different primes via the product Galois representation $\prod \rho_{A,p}$.  Ideally, we expect  
\begin{equation}\label{eq:ind}
    \Im\big(\prod_p \rho_{A,p}\big)=\prod_p\Im(\rho_{A,p}), 
\end{equation}
which implies: for different choices of $p$, the events that $\rho_{A,p}(\Frob_v)$ lands in a $\F_p$-rational Borel subgroup are independent. Hence, 
Theorem~\ref{thmA} follows (recall that the density upper bound $c$ is independent of $p$). Actually, Serre showed that the independence \eqref{eq:ind} holds upon replacing $K$ with a finite extension (cf.~Theorem~\ref{thm:Se}(2)). This suffices for our argument.

As for Theorem~\ref{thmB}, we quantify the above strategy with the aid of the effective Chebotarev theorem of Lagarias--Odlyzko and the Selberg sieve. Some notable features:
\begin{itemize}
    \item Without the Mumford--Tate conjecture, the $p$-adic monodromy groups can be different\footnote{Recall that the Mumford--Tate conjecture is not known in general.} for each $p$. In order to control their sizes, we introduce a soft argument based on the Weil bound and the classification of reductive groups over algebraically closed fields (cf.~the proof of Lemma~\ref{lem:WeilBound}).
    \item Since we are only sieving using primes which split completely in $M$, a lower bound on the number of such primes - independently of $M$ - is necessary. This is achieved via a simple \emph{a priori} upper bound on the discriminant of a Frobenius field $F(A,v)$ in terms of $\mathbf{N}_{K/\Q}v$ (cf.~Lemma~\ref{lem:1}).
\end{itemize}
To refine exponents naturally  appearing in the analysis, we make the following improvements:
\begin{itemize}
    \item Instead of the $p$-adic monodromy group, we work with its semisimple quotient, reducing the sizes of the relevant Galois groups. This is a generalization of the $\mathrm{PGL}_2$-reduction method used previously.
    \item Well-known results of Murty--Murty--Saradha give better error terms in the effective Chebotarev theorem when certain subgroups of the Galois group satisfy the Artin holomorphy conjecture (AHC) (cf.\ Theorem~\ref{thm:Res}). In our case, the natural choice is a Borel subgroup of the monodromy group, and we make an elementary group-theoretic observation that it satisfies the AHC.
\end{itemize}

 The above relies on a fundamental result of Deligne \cite{De}:  the Mumford--Tate group ``contains'' all $p$-adic monodromy groups (cf.~Theorem~\ref{thm:De}).

\subsection*{The prior work and prospects}
Our study is inspired by the Lang--Trotter conjecture \cite{LT} and 
a consideration of Serre for Hecke eigenvalues of elliptic newforms \cite[\S7]{S1}. For elliptic curves, the result was first stated\footnote{albeit without proof or explicit exponent} by Serre \cite[\S8.2]{S1}. Its various explicit forms have appeared in the work of Cojocaru--Fouvry--Murty \cite{CFM}, Cojocaru--David \cite{CD2}, Zywina \cite{Z}, and Kulkarni--Patankar--Rajan \cite{KPR}, among others. Cojocaru--David \cite{CD1} consider the analogous question for Drinfeld modules.

For generic abelian varieties, Theorem~\ref{thmA} was first established by Bloom \cite{Bl}. His method is different and does not seem amenable to quantitative refinements as in Theorem~\ref{thmB}. 
As for general abelian varieties, the only known result seems to be that of Khare \cite{Kh}: for a non-CM simple abelian variety, the set of places whose Frobenius fields equal a fixed number field 
cannot be the complement of a finite set of places.

Our result applies to any abelian variety up to a base change. In the generic case, no base change is needed, and it gives a better exponent than Bloom \cite{Bl}, assuming only the GRH. The qualitative part of the argument generalises to function fields, and we expect that a suitable modification yields the quantitative version. Moreover, our method seems to apply to any compatible system of Galois representations for which Frobenius semisimplicity holds (cf.~Remark~\ref{remark:cmp}).

Compared to previous works, our approach is closer in spirit to the square sieve method, which first appeared in \cite{CFM}. However, by working with abstract reductive groups, we replace all the work in estimating conjugacy class sizes with soft arguments, and we can identify power savings which are not apparent from the explicit matrix-based considerations. It would be interesting to see if our approach can be combined with the mixed Galois representation approach to obtain a further sharpening of Theorem~\ref{thmB}.

Finally, to remove Hypothesis \ref{hyp:mc}, we would need to study the structure of semisimple elements in a disconnected finite reductive group. We plan to return to this question in the near future.

\subsection*{Plan}
In section \ref{s:pre} we describe preliminaries regarding Galois representations, monodromy groups, and large image theorems.
In section \ref{sec:Grp} we present preliminaries regarding finite redcutive groups\footnote{As pointed out by the referee, these results may be well-known to specialists (cf.~\cite{JKZ,LR0}). Due to the lack of a precise reference in the literature (see Remark~\ref{rmk:prior}), we include details.} and introduce the key notion of a bounding set (cf.~Definition~\ref{def:B}). 
In section \ref{sec:Main} we prove the main theorems. 
\begin{thank} 
We thank Brian Conrad, Chun-Yin Hui, Chandrashekhar Khare, Vijay Patankar, Will Sawin and Victor Wang for helpful communications. We are grateful to the referee for valuable comments and suggestions. 
\end{thank}

\section{Backdrop}\label{s:pre}
We describe some preliminaries regarding $p$-adic Galois representations associated to an abelian variety over a number field. The reader may refer to \cite{HL0,P,W} for some details.

Let the setting be as in section \ref{s:Int}. 
In particular, $A$ is an abelian variety over a number field $K$ of dimension $g$. 
Put $\Delta_A=\mathbf{N}_{K/\Q}\mathfrak{N}$, so $A$ has good reduction away from places of $K$ which lie above a prime dividing $\Delta_A$.

\subsection{Galois representations}\label{ss:MT} 
For a prime $p$, let $T_{p}A$ be the $p$-adic Tate module of $A$. Let 
$$\rho_{p}: G_{K} \rightarrow \Aut_{\Z_{p}}T_{p}A$$
be the associated $p$-adic Galois representation. 

\begin{thm}\label{frob}
Let $v$ be a finite place of $K$ such that 
 $v \ndivide p\Delta_A$. 
\begin{enumerate}
  \item[(1)] (Serre--Tate) The Galois representation $\rho_{p}$ is unramified at $v$.
  \item[(2)] (Weil) The characteristic polynomial of $\rho_{p}({\rm Frob}_{v})$ has integral coefficients and is independent of $p$. Moreover, all of its roots have complex absolute value $(\mathbf{N}_{K/\Q}v)^{\frac{1}{2}}$.
\end{enumerate}
\end{thm}

For a set $T$ of primes, put
\[
  \rho^T=\prod_{p\notin T}\rho_p,\quad \rho_T=\prod_{p\in T}\rho_{p} .
\]
For $N$ a positive integer, put $\rho^N=\rho^{T(N)},\ \rho_N=\rho_{T(N)}$, where $T(N)$ is the set of primes dividing $N$.

\subsubsection{Monodromy groups}
Let $$\Gamma_p=\rho_p(G_K)\subseteq\Aut_{\Z_p}T_pA ,$$ and $\G_p$ denote its Zariski closure. This is a linear algebraic group over $\Z_p$. The generic fibre of $\G_p$ is the usual $p$-adic monodromy group. Throughout this paper, we make the following assumption.

\begin{hypothesis}\label{hyp:mc}
    The group $\G_p$ is connected for a prime $p$.
\end{hypothesis}
\begin{remark}\label{rmk:Se}
    A result of Serre \cite[No.~133, Corollaire~p.~15]{Serre:Ceuvres-IV} shows that the component group $\G_p/\G_p^\circ$ is independent of the prime $p$. So the hypothesis holds for all primes $p$ as soon as it holds for one prime. It is also proven in \emph{op.~cit.~}that there exists a finite extension $K^\mathrm{conn}/K$ such that this hypothesis holds for the base change $A_{/K^\mathrm{conn}}$.
\end{remark}

We will use the following key results on the monodromy group. 
\begin{thm}\label{thm:Se}\leavevmode
\begin{enumerate}
  \item[(1)] If $p$ is sufficiently large, then $\G_p$ is reductive.
  \item[(2)] There exists a finite Galois extension $L/K$ and an integer $N$ such that
  \[
    \rho^{N}(G_L)=\prod_{p\nmid N}\rho_p(G_L)
  \]
\end{enumerate}
\end{thm}
Part (1) is due to Larsen--Pink (cf.~the proof of \cite[Theorem 3.2]{LP}). Part (2) is another result of Serre \cite[No.~138, Th\'{e}or\`{e}me~p.~56]{Serre:Ceuvres-IV}
(see also~\cite{I,S2}). The reference does not include the Galois requirement, but we can take the Galois closure, which preserves independence away from a finite set of primes.

\subsubsection{Galois image}
This subsection will recall some large image results for non-CM abelian varieties due to Hui--Larsen \cite[Theorem~1.3]{HL0}. An essentially equivalent result was obtained earlier by Wintenberger \cite[Th\'eor\`eme~2]{W}, but this formulation is more convenient for us.

For a prime $p$ such that $\G_p$ is reductive, let $\G_p^\mathrm{ss}$ denote the quotient of $\G_p$ by its radical, and $\Gamma_{p}^{\rm ss}$ the image of $\Gamma_p$ in $\G_p^\mathrm{ss}$. Let $\pi:\G_p^\mathrm{sc}\to\G_p^\mathrm{ss}$ be the algebraic universal cover. This is a finite morphism. 
\begin{thm}\label{thm:HL}
Let $A$ be an abelian variety over a number field $K$. Then there exists a constant $c_{A,K} \in \Z$ such that for any prime $p > c_{A,K}$, we have $$\Im(\G_p^\mathrm{sc}(\Z_p)\to\G_p^\mathrm{ss}(\Z_p))\subseteq\Gamma_p^\mathrm{ss}.$$
\end{thm}
\begin{proof}
    The aforementioned result of Hui--Larsen is that for $p\gg 0$, the group $\G_p(\Q_p)$ is unramified, and $\pi^{-1}(\Gamma_p^\mathrm{ss})$ is a hyperspecial maximal compact subgroup of $\G_p^\mathrm{sc}(\Q_p)$. 
    
    It is not easy to extract the integral model directly from the proof, so we conclude in an indirect way. We have the inclusions
    \[
        \pi^{-1}(\Gamma_p^\mathrm{ss})\subseteq\pi^{-1}(\G_p^\mathrm{ss}(\Z_p))\subseteq\G_p^\mathrm{sc}(\Z_p)
    \]
    The first is by definition. The second holds since $\pi$ is finite, hence proper. The final group is also a hyperspecial maximal compact subgroup of $\G_p^\mathrm{sc}(\Q_p)$, and so all three groups above are equal.
\end{proof}

\begin{remark}\label{remark:cmp}
    If we have a compatible system of semisimple Galois representations coming from geometry, then Theorem~\ref{thm:Se} still holds. Larsen showed that Theorem~\ref{thm:HL} holds for a density one set of primes $p$ \cite[Theorem 3.17]{Lr}. The exceptional set in that proof is again cut out using Chebotarev sets. Therefore, it's likely that our approach works in such generality, though we have not checked the details carefully. 
\end{remark}

Define the reduced residual representation
\[
  \overline{\rho}^\mathrm{ss}_p:G_K\to\G_p^\mathrm{ss}(\F_p),
\]
Its image is also the image of $\Gamma_p^\mathrm{ss}$ under the reduction map $\G_p^\mathrm{ss}(\Z_p)\to\G_p^\mathrm{ss}(\F_p)$. Denote it by $\overline{\Gamma}_p^\mathrm{ss}$. 
Put $$I_p=\Im(\G_p^\mathrm{sc}(\F_p)\to\G_p^\mathrm{ss}(\F_p)).$$ 
A simple consequence of Theorem \ref{thm:HL} is the following.

\begin{cor}\label{cor:HL} 
%In the setting of Theorem~\ref{thm:HL}
Let $A$ be an abelian variety over a number field $K$ and $c_{A,K}\in \Z$ be as in Theorem~\ref{thm:HL}. Then 
%there exists a constant $c_{A,K} \in \Z$ such that 
     for any prime $p>c_{A,K}$, the image $\overline{\Gamma}_p^\mathrm{ss}$ contains $I_p$.
\end{cor}

\begin{remark}
  %In both the above corollary and Theorem~\ref{thm:Se}, the bound on the prime $p$ 
  The lower bound $c_{A,K}$ in the above corollary and Theorem~\ref{thm:Se} depends only on the number field $K$ and the isogeny class of $A$. It seems likely that $c_{A,K}$ can be expressed as an explicit function of $K$ and the conductor of $A$ following the strategy outlined in \cite[Remarque 2.2]{W}.
  %but we have not pursued this further.
\end{remark}

\subsection{A characterisation of CM abelian varieties}

\begin{prop}\label{prop:non-ab}
Let $A$ be a simple abelian variety over a number field $K$. 
If the associated $p$-adic monodromy group $\G_{p/\Q_p}$ is abelian for one prime $p$, then it is so for all primes, and $A$ has CM.
\end{prop}
\begin{proof}
Suppose that ${\bf G}_{p/\Q_p}$ is abelian. The $p$-adic Galois representation $\rho_{p}$
is semisimple by Faltings, and rational by construction. Hence, 
this Galois representation arises from an arithmetic Hecke character over $K$ by a result of Henniart \cite{He} (cf.~\cite[Theorem 2 in III-13]{S'}).
Considering infinity type of the Hecke character, it follows that $\rho_p$ corresponds to a CM abelian variety $A'$. 
Note that $A$ and $A'$ are isogenous by Falting's isogeny theorem. Hence $A$ has CM, and the assertion follows.
\end{proof}

\section{Results on finite reductive groups}\label{sec:Grp}
Let $\G$ be a connected reductive group defined over $\F_p$. In this section we will give a volume upper bound for the union of certain conjugacy classes.

\subsection{Tori and Weyl groups}
We recall some basic results about maximal tori of $\G$. A detailed exposition can be found in \cite[Section~3.3]{Car}.

Let $W_G$ be the absolute Weyl group of $\G$. It carries an action of the Galois group $\Gal(\overline{\F}_p/\F_p)$. Let $\gamma$ be the action corresponding to the Frobenius element. The general theory of forms gives a bijection
\[
  \{\text{Conjugacy class of maximal tori defined over }\F_p\}\longleftrightarrow\{\gamma\text{-conjugacy classes in }W_{G}\}.
\]
Let $\bT$ be a maximal torus in $\G$. Define its Weyl group by
\[
  W(\mb{G},\mb{T})=N_{\G(\F_p)}(\bT)/\bT(\F_p).
\]
It can also be described as the stabilizer of the $\gamma$-conjugacy class corresponding to $\mb{T}$ under the above bijection. In particular, the $\gamma$-twisted class equation gives
\begin{equation}\label{eqn:Class}
  \sum_{\mb{T}}\frac{1}{\abs{W(\mb{G},\mb{T})}}=1
\end{equation}
where the sum is over a set of representatives for the conjugacy classes of maximal tori in $\G$.

\begin{lm}\label{lem:Weyl}
  If $\G$ is not abelian, then $W(\mb{G},\mb{T})$ is non-trivial for any maximal torus $\bT$.
\end{lm}
\begin{proof}
  By the class equation~\eqref{eqn:Class}, if the lemma holds for one choice of $\bT$, then it holds for all choices of $\bT$. Let $\mb{B}$ be a Borel subgroup of $\G$ defined over $\F_p$. Its existence is guaranteed by Lang's theorem \cite{L}. Let $\mb{S}$ be a maximal $\F_p$-split torus contained in $\mb{B}$, and let $\mb{T}$ be its centralizer. This is a Levi subgroup of $\mb{B}$, and hence a maximal torus. In addition, $N_\mb{G}(\mb{S})\subseteq N_\mb{G}(\mb{T})$, so the Weyl group $W(\mb{G},\mb{T})$ contains the relative Weyl group $W(\G,\mb{S})$.\footnote{We in fact have equality here, but this will not be needed.}

  Since $\G$ is not abelian, its derived subgroup $\G^\der$ is non-trivial. By Lang's theorem, $\G^\der$ has a Borel subgroup, which implies that $\G^\der$ has positive rank \cite[Corollaire~4.17]{BT}. The second paragraph of the proof of \cite[Th\'eor\`eme 5.3]{BT} then shows that the relative Weyl group $W(\mb{G},\mb{S})$ is non-trivial, so $W(\mb{G},\mb{T})$ is also non-trivial for this choice of $\bT$.
\end{proof}

\subsection{Conjugacy classes}
Given a maximal torus $\bT$, let $C_{\bT}$ denote the set of elements in $\G(\F_p)$ which are conjugate to an element of $\bT(\F_p)$. We will also use the superscript ``reg'' to denote the subset of regular semisimple elements. The volume is the counting measure normalized so that $\vol(\G(\F_p))=1$. 
\begin{prop}\label{prop:VolBound}
  There exists a constant $C$ depending only on the (absolute) rank of the group $\G$ such that the following holds for all $p$
  \begin{enumerate}
    \item[(1)] $\vol(\G(\F_p)^\reg)>1-Cp^{-1}$, 
    \item[(2)] If $\bT$ is a maximal torus, then
    \[
      \abs{\vol(C_\bT^\reg)-\frac{1}{\abs{W(\G,\bT)}}}<Cp^{-1}.
    \]
  \end{enumerate}
\end{prop}
\begin{proof}
Let $\bT$ be a maximal torus. The map $\G/\bT\times\bT\to\G,\ (g,t)\mapsto gtg^{-1}$ is finite of degree $\abs{W(\mb{G},\mb{T})}$ above the regular elements in its image. Indeed, if $gtg^{-1}=g't'(g')^{-1}$, then $g^{-1}g'$ conjugates $t'$ to $t$, so it conjugates $Z_\G(t')$ to $Z_\G(t)$. If $t$ is regular, then the centralizer is just $\bT$, so $g^{-1}g'\in N_\G(\bT)$. On $\F_p$-points, the image of this map is exactly $C_\bT$, so
\begin{equation}\label{eqn:CVol}
  \vol(C_{\bT}^\reg)=\frac{1}{\abs{W(\mb{G},\mb{T})}}\cdot\frac{\abs{\bT(\F_p)^\reg}}{\abs{\bT(\F_p)}}.
\end{equation}
Summing over a set of representatives of conjugacy classes of maximal tori, we get
\[
  \vol(\G(\F_p)^\reg)=\sum_{\bT}\frac{1}{\abs{W(\mb{G},\mb{T})}}\cdot\frac{\abs{\bT(\F_p)^\reg}}{\abs{\bT(\F_p)}}.
\]
We now give a coarse estimate for the right hand side.

Let $r=\rank\bT$. The set of non-regular elements in $\bT(\F_p)$ is the $\F_p$-points of a variety $\bT^\mathrm{nr}$ of dimension $r-1$. When base changed to $\overline{\F}_p$, it can be described as the subset of $\G_{m,\overline{\F}_p}^r$ where at least two of the entries are equal. Let $E_{/\Z}$ be the subscheme of $\G_{m,\Z}^r$ defined this way, then we see that $\bT^{\mathrm{nr}}$ is a form of $E\times_\Z\F_p$. It follows from the Weil bound\footnote{One may also proceed by an elementary argument: counting $\mb{T}(\F_p)^\reg$ using inclusion-exclusion, but the details are messy.} that there exist a constant $C$ (depending only on $r$) such that 
\begin{equation}\label{eqn:TorusReg}
  \left|\frac{\abs{\bT(\F_p)^\reg}}{\abs{\bT(\F_p)}}-1\right|\leq Cp^{-1}
\end{equation}
for all $p$ (cf.\ the proof of Lemma~\ref{lem:WeilBound}).

This estimate together with equation~\eqref{eqn:CVol} immediately gives part (2) of the proposition. Using equation~\eqref{eqn:Class}, we see that
\[
  \vol(\G(\F_p)^\reg)\geq\sum_{\bT}\frac{1}{\abs{W(\G,\bT)}}\cdot(1-Cp^{-1})=1-Cp^{-1}.
\]
This proves (1).
\end{proof}

\begin{defn}\label{def:B}
    Let $\G$ be a reductive group over $\F_p$. Its \emph{bounding set} is the set of elements in $\G(\F_p)$ which lie in a Borel subgroup defined over $\F_p$.
\end{defn}
A consequence of Proposition \ref{prop:VolBound} is the following. 
\begin{cor}\label{cor:UpperBound}
  Let $\cB$ be the bounding set of $\G$. There exists an integer $N$ depending only on the rank of $\G$ such that if $\G$ is non-abelian and $p>N$, then
  \[
    \vol(\cB)<\frac{3}{4}.
  \]
\end{cor}
\begin{proof}
  Let $\bT$ be a maximal torus in a Borel subgroup $\mathbf{B}$. Let $x\in\cB$ be regular semisimple. We will show that $x$ is conjugate to an element of $\bT$. Indeed, all Borel subgroups are conjugate \cite[Th\'eor\`eme 4.13(b)]{BT}, so we may assume $x\in\mathbf{B}$. Any maximal torus of $\mathbf{B}$ containing $x$ is a Levi component, so it is conjugate to $\bT$ \cite[Proposition 4.7]{BT}. It follows that $\cB$ is contained in the union of non-regular semisimple elements and $C_\bT$. Therefore,
  \[
    \vol(\cB)\leq\vol(\G(\F_p)-\G(\F_p)^\reg)+\vol(C_\bT^\reg)<2Cp^{-1}+\frac{1}{\abs{W(\G,\bT)}}
  \]
by Proposition \ref{prop:VolBound}. 
 Lemma~\ref{lem:Weyl} shows that $W(\G,\bT)$ is non-trivial, so its order is at least 2. By taking $p$ sufficiently large, the right hand side can be made smaller than any real number greater than $\frac{1}{2}$.
\end{proof}

\subsection{Central isogeny}\label{ss:Isog}
This subsection refines Proposition~\ref{prop:VolBound} and Corollary~\ref{cor:UpperBound}. This refinement is the key ingredient which allows us to apply the large image results of Theorem~\ref{thm:HL} and Corollary~\ref{cor:HL}.

Suppose $\G$ is semisimple, then it has an algebraic universal cover $\pi:\G^\mathrm{sc}\to\G$. Its kernel, denoted by $Z$, is a finite group scheme over $\F_p$ and contained in the centre of $\G^\mathrm{sc}$. Let $\bT$ be a maximal torus of $\G$, then $\bT^\mathrm{sc}:=\pi^{-1}(\bT)$ is a maximal torus in $\G^\mathrm{sc}$ containing $Z$. 

We have the following exact sequence of groups
\[
\begin{tikzcd}
    1\rar & Z(\F_p)\rar\dar[equals] & \bT^{\mathrm{sc}}(\F_p)\rar\dar[hook] & \bT(\F_p)\rar\dar[hook] & \mathrm{H}^1(\F_p,Z)\rar\dar[equals] & 1\\
    1\rar & Z(\F_p)\rar & \G^{\mathrm{sc}}(\F_p)\rar & \G(\F_p)\rar["\delta"] & \mathrm{H}^1(\F_p,Z)\rar & 1.
\end{tikzcd}
\]
The final 1 in both rows are by Lang's theorem. Note that \emph{a priori}, the connecting morphisms are between pointed sets, but since $Z$ is central, they are actually group homomorphisms (cf.~\cite[\S I.5.6]{Serre:GalCoh}). 

\begin{lm}\label{lm:IsogCount}
We have 
  $\abs{\G^{\mathrm{sc}}(\F_p)}=\abs{\G(\F_p)}$.
\end{lm}
\begin{proof}
Since $Z(\overline{\F}_p)$ is finite, its Herbrand quotient is 1. Hence the assertion follows by considering the 
bottom row of the above diagram. We remark that this is also a special case of \cite[Chapter V, Proposition 16.8]{BorelLAG}.
\end{proof}

The main result of this section is that the various conjugacy classes we have considered are approximately equally distributed in the fibres of $\delta$. More precisely, we have the following.

\begin{prop}\label{cor:CosetUpperBound}
  Let $\G$ be a non-abelian semisimple group. Let $I=\Im(\G^\mathrm{sc}(\F_p)\to\G(\F_p))$.
  \begin{enumerate}
    \item[(1)] There exists a constant $C$ depending only on the rank of $\G$ such that for any maximal torus $\bT$ and $g\in\G(\F_p)$, we have
    \[
        \bigg|\frac{\abs{C_\bT^\reg\cap gI}}{\abs{I}}-\frac{1}{\abs{W(\G,\mb{T})}}\bigg|<Cp^{-1}.
    \]
    \item[(2)] There exists an integer $N$ depending only on the rank of $\G$ such that for all $p>N$, we have
    \[
        \frac{1}{2|W(\G,\bT)|}<\frac{\abs{\cB\cap gI}}{\abs{I}}<\frac{3}{4}
    \]
    where $\cB$ is the bounding set for $\G$.
  \end{enumerate}
\end{prop}
\begin{proof}
  Let $n=[\G(\F_p):I]=\abs{\mathrm{H}^1(\F_p,Z)}=\abs{Z(\F_p)}$. The group $Z$ is a subgroup of the center of a semisimple simply connected group over $\overline{\F}_p$. Since the rank is fixed, there are a finite number of such groups, so $n$ is bounded from above, independently of $\G$. 
  
  From the above diagram, $I=\ker\delta$ is a normal subgroup with an abelian quotient, so if two elements of $\G(\F_p)$ are conjugate, then they lie in the same coset of $I$. Moreover, $\delta$ is still surjective when restricted to $\bT(\F_p)$, so it identifies $\bT(\F_p)/(\bT(\F_p)\cap I)$ with $\mathrm{H}^1(\F_p,Z)$. Fix $g\in\G(\F_p)$, then we get
  \[
    \abs{\bT(\F_p)\cap gI}=\frac{\abs{\bT(\F_p)}}{\abs{\mathrm{H}^1(\F_p,Z)}}.
  \]
  The conjugation action $\G/\bT\times\bT\to\G,\ (x,t)\mapsto xtx^{-1}$, as in the proof of Proposition~\ref{prop:VolBound}, restricts to an action on the coset $gI$, so
  \begin{align*}
    \abs{C_\bT^\mathrm{reg}\cap gI}&=\frac{1}{\abs{W(\G,\bT)}}\cdot\frac{\abs{\G(\F_p)}}{\abs{\bT(\F_p)}}\cdot\abs{\bT(\F_p)^\mathrm{reg}\cap gI}\\
    &=\frac{\abs{I}}{\abs{W(\G,\bT)}}\cdot\frac{\abs{\bT(\F_p)^\reg\cap gI}}{\abs{\bT(\F_p)\cap gI}}.
  \end{align*}
  Therefore, we have
  \[
    \bigg|\frac{\abs{C_\bT^\mathrm{reg}\cap gI}}{\abs{I}}-\frac{1}{\abs{W(\G,\bT)}}\bigg|=\frac{\abs{(\bT(\F_p)-\bT(\F_p)^\reg)\cap gI}}{\abs{\bT(\F_p)\cap gI}}\leq\frac{\abs{\bT(\F_p)-\bT(\F_p)^\reg}}{\abs{\bT(\F_p)}/n}\leq nCp^{-1}
  \]
  where $C$ is the constant from the equation~\eqref{eqn:TorusReg}.
  
  For the second part, take $\bT$ to be the maximal torus in a Borel subgroup as before. Let $C$ be the constant in Proposition~\ref{prop:VolBound}, then
  \begin{align*}
    \frac{\abs{\cB\cap gI}}{\abs{I}}&\leq\frac{\abs{\G(\F_p)-\G(\F_p)^\reg}}{\abs{I}}+\frac{\abs{C_\bT^\mathrm{reg}\cap gI}}{\abs{I}}\\
    &<C(n+1)p^{-1}+\frac{1}{\abs{W(\G,\bT)}}. 
  \end{align*}
  On the other hand,
  \[
    \frac{\abs{\cB\cap gI}}{\abs{I}}\geq \frac{\abs{C_\bT^\mathrm{reg}\cap gI}}{\abs{I}}\geq\frac{1}{|W(\G,\bT)|}-nCp^{-1}
  \]
  The same argument as in Corollary~\ref{cor:UpperBound} gives the desired bounds.
\end{proof}

\begin{remark}\label{rmk:prior}
    Some versions of the above counting results for tori have previously appeared in literature (see~\cite[\S3.4]{JKZ} and \cite[\S3]{LR0}). For example, our Propositions~\ref{prop:VolBound}(2) and \ref{cor:CosetUpperBound}(1) coincide with Propositions 4.1 and 4.6 of \cite{JKZ} respectively. However, the literature does not seem to contain finer results such as Corollary \ref{cor:UpperBound} and 
    Proposition \ref{cor:CosetUpperBound}(2), which are necessary for our later arguments. 
    
\end{remark}

\section{Proof of Main Theorem}\label{sec:Main}
We now prove the main results. The first subsection sets up some notation. Then the second contains a short proof of the qualitative version (cf.~Theorem~\ref{thmA}). Finally, the third subsection is dedicated to the quantitative version (cf.~Theorem~\ref{thmB}).

\subsection{Preliminaries}\label{ss:Prep}
\subsubsection{Setting}
Let $A$ be a simple abelian variety of dimension $g$ defined over a number field $K$ such that $A_{/\bar{K}}$ does not have CM. Let $\mathfrak{N}$ be the conductor of $A$. Let $M$ be a number field and
$$
S_{A,M}=\{v \in \Sg_{K}\,|\,v \ndivide \mathfrak{N}, F(A,v)\cong M\}. 
$$
For $X$ a positive real number, define
\[
    S_{A,M}(X)=\{v\in S_{A,M}\,|\,\mathbf{N}_{K/\Q}v\leq X\}. 
\]
Our goal is to bound $S_{A,M}(X)$ from above, 
supposing ~Hypothesis~\ref{hyp:mc}. 

Let $L$ be a finite Galois extension of $K$ such that the conclusion of Theorem~\ref{thm:Se}(2) holds, and let $N$ be the bound given therein. Increase $N$ so that for all $p>N$, the following holds: 
\begin{enumerate}
    \item[(1)] The monodromy group $\G_p$ is reductive, cf.\ Theorem~\ref{thm:Se}(1).
    \item[(2)] The large image result in Corollary~\ref{cor:HL} holds for $A_{/L}$
    \item[(3)] The density estimate in Proposition~\ref{cor:CosetUpperBound}(2) holds for all groups of rank at most $2g$.
\end{enumerate}
In view of Corollary~\ref{cor:HL}, this $N$ depends only on $A$ and $K$. Let $\cP$ be the set of all primes greater than $N$ which split completely in $M$. 

\subsubsection{Preliminary lemmas}
\begin{lm}\label{lm:Brl}
    Pick $p\in \cP$ and $v\in S_{A,M}$ which does not lie above $p$. Then the Frobenius $\overline{\rho}_p^\mathrm{ss}(\Frob_v)\in\G_p^\mathrm{ss}(\F_p)$ lives in a Borel subgroup defined over $\F_p$.
\end{lm}
\begin{proof}
    Let $F_v=\rho_p^\mathrm{ss}(\Frob_v)\in\G_p^\mathrm{ss}(\Z_p)$. By Faltings' theorem, this is a semisimple element in $\G_p^\mathrm{ss}(\Q_p)$. Moreover, our assumptions on $p$ and $v$ together imply that $F_v$ lies in a split torus, and hence in a minimal parabolic subgroup of $\G_p^\mathrm{ss}(\Q_p)$, which is necessarily a Borel subgroup since $\G_p^\mathrm{ss}$ is unramified. The scheme of Borel subgroups over $\Z_p$ is proper \cite[XXII, 5.8.3(i)]{SGA3}, so we can extend this Borel subgroup to one defined over $\Z_p$. Its fibre over $\F_p$ is a Borel subgroup of $\G_p^\mathrm{ss}(\F_p)$ containing the reduction of $F_v$.
\end{proof}

\begin{lm}\label{lm:Coset}
    The image of the Galois representation
    \[
        \overline{\rho}^\mathrm{ss}_{\cP}:G_K\to\prod_{p\in\cP}\G_p^\mathrm{ss}(\F_p)
    \]
    is a union of cosets of $\prod_{p\in\cP}I_p$, where 
    $I_p$ is defined as in Corollary~\ref{cor:HL}.
\end{lm}
\begin{proof}
By the choice of $L$, we have
\[
    \overline{\rho}^{\mathrm{ss}}_{\cP}(G_L)=\prod_{p\in\cP}\overline{\rho}_p^{\mathrm{ss}}(G_L)\subseteq\prod_{p\in\cP}\G_p^\mathrm{ss}(\F_p).
\]
Item (2) of the choice of $\cP$ implies that $\overline{\rho}_p^\mathrm{ss}(G_L)$ contains the subgroup $I_p$ for all $p\in\cP$. Therefore,
\[
    \overline{\rho}_{\cP}^\mathrm{ss}(G_K)\supseteq\overline{\rho}^{\mathrm{ss}}_{\cP}(G_L)\supseteq\prod_{p\in\cP}I_p.
\]
In other words, $\overline{\rho}_{\cP}^\mathrm{ss}(G_K)$ is a subgroup of $\prod_{p\in\cP}\G_p^\mathrm{ss}(\F_p)$ which contains the subgroup $\prod_{p\in\cP}I_p$.
\end{proof}

\subsection{Proof of Theorem~\ref{thmA}}
Let $d$ be a square-free product of primes in $\cP$. Let $\cF_d$ be the union of all conjugacy classes $\overline{\rho}_d^\mathrm{ss}(\Frob_v)$ for $v\in S_{A,M}$ not dividing $d$. 

By Lemma~\ref{lm:Brl}, we have
\[
  \cF_d\subseteq\Big(\prod_{p|d}\cB_p\Big)\cap\overline{\rho}_d^\mathrm{ss}(G_K)
\]
where $\cB_p$ is the bounding set for $\G_p^\mathrm{ss}$ introduced in Definition~\ref{def:B}. The Chebotarev density theorem implies that 
\begin{equation}\label{eqn:density1}
    \mathrm{ud}(S_{A,M})\leq\frac{\abs{\cF_d}}{\abs{\overline{\rho}_d^\mathrm{ss}(G_K)}}\leq\frac{\big|\big(\prod_{p|d}\cB_p\big)\cap\overline{\rho}_d^\mathrm{ss}(G_K)\big|}{\abs{\overline{\rho}_d^\mathrm{ss}(G_K)}}. 
\end{equation}
To estimate the right hand side,  we utilise the following decomposition
\[
    \overline{\rho}_d^\mathrm{ss}(G_K)=\bigsqcup_{i=1}^s\Big(g_i\prod_{p|d}I_p\Big)=\bigsqcup_{i=1}^s\Big(\prod_{p|d}g_{i,p}I_p\Big)
\]
given by Lemma~\ref{lm:Coset}. In view of Proposition~\ref{prop:non-ab}, $\G_{p}^{\rm ss}$ is non-abelian\footnote{The generic fibre is abelian if and if the special fibre is, since $\G_p^\mathrm{ss}$ is reductive over $\Z_p$.}. Therefore, Proposition~\ref{cor:CosetUpperBound} can be applied, giving the upper bound
\begin{equation}\label{eq:VolBound}
    \bigg|\Big(\prod_{p|d}\cB_p\Big)\cap\overline{\rho}_d^\mathrm{ss}(G_K)\bigg|=\sum_{i=1}^s\bigg|\prod_{p|d}\big(\cB_p\cap g_{i,p}I_p\big)\bigg|\leq s\prod_{p|d}\Big(\frac{3}{4}\abs{I_p}\Big)=\Big(\frac{3}{4}\Big)^{\omega(d)}\abs{\overline{\rho}_d^\mathrm{ss}(G_K)}
\end{equation}
where $\omega(d)$ is the number of prime divisors of $d$.

Combining equations \eqref{eqn:density1} and \eqref{eq:VolBound}, we get
\[
    \mathrm{ud}(S_{A,M})\leq\Big(\frac{3}{4}\Big)^{\omega(d)}
\]
for any $d$ which is a square-free product of primes in $\cP$. Since $\cP$ is infinite, this implies that 
$$\mathrm{ud}(S_{A,M})=0.$$

\subsection{Proof of Theorem~\ref{thmB}}
We will now use the effective Chebotarev theorem of Lagarias--Odlyzko and the Selberg sieve to obtain a power saving upper bound for the number of places with a given Frobenius field $M$. This is possibly the approach suggested by Serre in \cite[\S8.2]{S1}.

In this section we will write $f(X)=O(g(X))$ or $f(X)\ll g(X)$ to mean the existence of an absolute constant $c$ such that $\abs{f(X)}\leq c\abs{g(X)}$ for all $X$. If the constant is allowed to depend on some other object, we will indicate so in the subscripts. All constants are in fact effectively computable.

\begin{remark}
    This section will be conditional on the GRH. Unconditionally, we expect that the standard methods lead to an upper bound of the form $O\big(\frac{X}{(\log X)^\alpha}\big)$ for some explicit $\alpha>1$, but we have not worked out the details.
\end{remark}

\subsubsection{Sieving setup}
In this subsection, we recast the problem into a form where the Selberg sieve can be directly applied. There are notational complications since we are only assuming potential independence of the Galois images at different primes $p$. No important idea is lost if one assumes $L=K$, i.e.~the family of Galois representations over $K$ has independent image.

Fix a positive real number $X$. Let $\Sg_K(X)$ be the set of all finite places of $K$ of norm at most $X$ where $A$ has good reduction. Recall that in \S\ref{ss:Prep}, we have chosen a finite Galois extension $L/K$ and an infinite set of primes $\cP$ defined by a splitting condition.

Let $\cR$ be the set of square-free products of primes in $\cP$. Let $d\in\cR\cup\{\infty\}$. In the infinite case, the condition ``$p|d$'' should be interpreted as ``$p\in\cP$''. 
Let $G_d=\prod_{p|d}\G_p^\mathrm{ss}(\F_p)$, so we have a Galois representation
\[
    \overline{\rho}_d^\mathrm{ss}:G_K\to G_d.
\]
Denote its image by $G_d'$ and the fixed field of its kernel by $K_d$, so we have an isomorphism $$\Gal(K_d/K)\simeq G_d',$$ and $K_d$ is the compositum of all $K_p$ for $p|d$. Define $H_d$ and $L_d$ similarly using the restriction $\overline{\rho}_d^\mathrm{ss}|_{G_L}$. By the choice of $L$, we have $H_d=\prod_{p|d}H_p$. In these notations, we have
\[
  I_d:=\prod_{p|d}I_p\subseteq H_d\subseteq G_d'\subseteq G_d
\]
where the first inclusion is due to Corollary~\ref{cor:HL}. 

Here is a summary of the above definitions
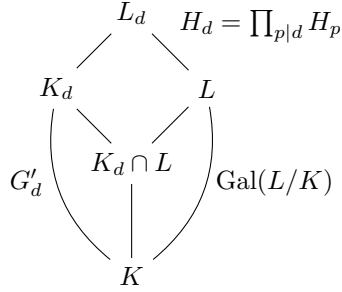
\begin{center}
  \begin{tikzpicture}
  \node (K) at (0, 0) {$K$};
  \node (I) at (0, 1.5) {$K_d\cap L$};
  \node (L) at (1,2.5) {$L$};
  \node (Kd) at (-1, 2.5) {$K_d$};
  \node (Ld) at (0,3.5) {$L_d$};

  \draw (I) -- (Kd);
  \draw (I) -- (L);
  \draw (K) -- (I); % node [pos=0.5, left] {$\cI$};
  \path [bend left] (K) edge (Kd);
  \path [bend right] (K) edge (L);
  \draw (L) -- (Ld) node [pos=0.5, above right] {$H_d=\prod_{p|d}H_p$};
  \draw (Kd) -- (Ld);
  \node at (-1.4,1.25) {$G_d'$};
  \node at (1.9,1.25) {$\Gal(L/K)$};
  \end{tikzpicture}
\end{center}
Let $\cI=G_\infty'/H_\infty$, then from the above diagram, $\cI=\Gal(K_\infty\cap L/K)$ is finite. Fix an integer $\mathtt{d}$ such that $K_\mathtt{d}\cap L=K_\infty\cap L$, so if $\mathtt{d}|d$, then $G_d'/H_d\simeq\cI$. Write down a coset decomposition
\[
  G_\infty'=\bigsqcup_{i\in\cI}g_i H_\infty=\bigsqcup_{i\in\cI}\prod_{p\in\cP}g_{i,p}H_p.
\]
For any $d$ as above, define $g_{i,d}=(g_{i,p})_{p|d}\in G_d'$, so we have a decomposition $G_d'=\bigcup_{i\in\cI}g_{i,d}H_d$. By comparing indices, this is a disjoint union if $\mathtt{d}|d$.

For each $p\in\cP$, let $\cB_p\subseteq G_p$ be the bounding set as in the previous section (cf.\ Definition~\ref{def:B}). Let $\cC_p=G_p-\cB_p$. For a general $d\in\cR\cup\{\infty\}$, define $\cC_d$ and $\cB_d$ as a product of the corresponding sets for $p$ dividing $d$. Let $\cC_d'=\cC_d\cap G_d'$ and $\cB_d'=\cB_d\cap G_d'$. For an index $i\in\cI$, define
\[
  \cB_d^{(i)}=\cB_d\cap\prod_{p|d}g_{i,p}H_p.
\]
So we have $\cB_d^{(i)}=\prod_{p|d}\cB_p^{(i)}$. Moreover, $\cB_d'=\bigcup_{i\in\cI}\cB_d^{(i)}$. Define the subset $\cC_d^{(i)}$ similarly. If $p$ is a prime, then $g_{i,p}H_p=\cB_p^{(i)}\sqcup\cC_p^{(i)}$.

\begin{remark}\label{rem:d}
  It could be the case that $\cB_d^{(i)}=\cB_d^{(i')}$ even though $i\neq i'$ in $\cI$. However, this cannot happen if $\mathtt{d}|d$, in which case $\cB_d'=\bigsqcup_{i\in\cI}\cB_d^{(i)}$. In other words, the purpose of the auxiliary $\mathtt{d}$ is to ensure that each element of $\cB_d'$ belongs to a unique coset.
\end{remark}

Let $C$ be any union of conjugacy classes in $G_d'$. Define
\[
  E_d(C)=\Big\{v\in \Sg_K(X)\,\Big|\,v\text{ does not lie above any }p|d,\ \overline{\rho}_d^\mathrm{ss}(\Frob_v)\in C\Big\}.
\]
The proof of Theorem~\ref{thmA} in the previous subsection then implies $S_{A,M}(X)\subseteq E_d(\cB_d')$. This form does not immediately yield a power saving upper bound. Instead, we re-write it in a different way.

\begin{prop}\label{prop:Sieve}
  We have an inclusion
  \[
    S_{A,M}(X)\subseteq\bigsqcup_{i\in\cI}\bigg(E_\mathtt{d}(\cB_{\mathtt{d}}^{(i)})-\bigcup_{p\in\cP^{(\mathtt{d})}}E_{\mathtt{d}p}(\cB_{\mathtt{d}}^{(i)}\times\cC_{p}^{(i)})\bigg)
  \]
  where $\cP^{(\mathtt{d})}$ consists of the primes in $\cP$ which do not divide $\mathtt{d}$.
\end{prop}
\begin{proof}
  Implicit in the statement is that $\cB_d^{(i)}$ is a union of conjugacy classes in $G_d'$ for all $d$. To see this, observe that $\cB_d'$ is a union of conjugacy classes, and $\cB_d^{(i)}$ is its intersection with a coset of $H_d$ in $G_d'$. It remains to observe that $H_d$ is a normal subgroup of $G_d'$, and the quotient is abelian, since it is a subquotient of the abelian group $G_d/I_d$ (cf.\ \S\ref{ss:Isog}). The same argument works for the second term.

  The containment is apparent: let $v\in S_{A,M}(X)$ and $g=(g_p)_{p\in\cP}=\overline{\rho}_\infty^\mathrm{ss}(\Frob_v)\in G_\infty'$. Then there exists a unique $i\in\cI$ such that $g\in g_i H_\infty$. Moreover, we know that $g_p\in\cB_p'$ for all $p\in\cP$, so $g_p\notin\cC_p'$. It follows that $(g_p)_{p|\mathtt{d}}\in\cB_\mathtt{d}^{(i)}$, but it is not in $\cB_\mathtt{d}^{(i)}\times\cC_p^{(i)}$ for any prime $p$.
\end{proof}

Each term in the disjoint union in Proposition \ref{prop:Sieve} is a sieving problem. We now recall a form of the Selberg sieve.

\begin{thm}\label{thm:SelbergSieve}
    Let $\mathscr{P}$ be a set of primes, and let $\mathscr{R}$ be the set of square-free products of primes in $\mathscr{P}$. 
    
    Let $\cA$ be a finite set. For each $p\in\mathscr{P}$, let $\cA_p\subseteq \cA$. For $d\in\mathscr{R}$, define $\cA_d=\bigcap_{p|d}\cA_p$. If $d=1$, set $\cA_1=\cA$. Suppose for all $d$, we can write
    \[
      \abs{\cA_d}=\beta_dc\,\mathrm{Li}(X)+R_d
    \]
    where $R_d$ is some real number, the function $d\mapsto\beta_d$ is multiplicative, and there exists constants $0<\underline{\beta}<\overline{\beta}<1$ such that $\underline{\beta}<\beta_p<\overline{\beta}$ for all $p\in\mathscr{P}$.
    
    Under these assumptions, for any positive real numbers $z$ and $\varepsilon$, we have
    \[
        \bigg|\cA-\bigcup_{p\in\mathscr{P}}\cA_p\bigg|\ll_{\underline{\beta},\overline{\beta},\varepsilon}\frac{c\,\mathrm{Li}(X)}{\pi_\mathscr{P}(z)}+\left(\frac{z^{1+\varepsilon}}{\pi_\mathscr{P}(z)}\right)^2\sum_{\substack{d_1,d_2\leq z\\ d_1,d_2\in\cR}}\frac{R_{[d_1,d_2]}}{d_1d_2}
    \]
    where $\pi_\mathscr{P}(z)=\abs{\{p\in\mathscr{P}\,|\,p\leq z\}}$.
\end{thm}
\begin{proof}
    This follows from a minor refinement to the proof of \cite[Theorem 7.2.1]{CM}, and we borrow their notations, with the obvious modification that their $X$ is $c\,\mathrm{Li}(X)$ in our case. It is easy to compute that $f(d)=\beta_d^{-1}$, $f_1(d)=\prod_{p|d}\frac{1-\beta_p}{\beta_p}$, and
    \[
        V(z)=\sum_{\substack{d\leq z\\d|P(z)}}\frac{\mu^2(d)}{f_1(d)}\geq\sum_{\substack{p\leq z\\p\in\mathscr{P}}}\frac{1}{f_1(p)}=\sum_{\substack{p\leq z\\p\in\mathscr{P}}}\frac{\beta_p}{1-\beta_p}\geq\frac{\underline{\beta}}{1-\underline{\beta}}\pi_{\cP}(z).
    \]
    The first term in the upper bound is the main term of the conclusion in \emph{loc.\ cit.}. For the second term, we start from the third displayed equation on page 122 and improve on the estimate $|\lambda_d|\leq 1$. From the second to last displayed equation on page 122, we have
    \begin{equation}\label{eq:SelbergError}
        |V(z)\lambda_d|\leq\prod_{p|d}\frac{1}{1-\beta_p}\cdot \sum_{\substack{t\leq\frac{z}{d}\\t\in\mathscr{R}}}\frac{\mu^2(t)}{f_1(t)}
    \end{equation}
    Fix an $\varepsilon>0$, then for all $t\in\mathscr{R}$,
    \[
      \frac{1}{f_1(t)}=\prod_{p|t}\frac{\beta_p}{1-\beta_p}\leq\left(\frac{\overline{\beta}}{1-\overline{\beta}}\right)^{\omega(t)}\ll_{\overline{\beta},\varepsilon} t^\varepsilon
    \]
    The same argument can be applied to the first term of equation~\eqref{eq:SelbergError}. Therefore,
    \[
      |V(z)\lambda_d|\ll_{\overline{\beta},\varepsilon} d^\varepsilon\sum_{\substack{t\leq\frac{z}{d}\\ t\in\mathscr{R}}}t^\varepsilon\leq d^{-1}z^{1+\varepsilon}
    \]
    Combined with the above lower bounds for $V(z)$, we get
    \[
      |\lambda_d|\ll_{\underline{\beta},\overline{\beta},\varepsilon} d^{-1}\frac{z^{1+\varepsilon}}{\pi_\mathscr{P}(z)}
    \]
    This concludes the proof.
\end{proof}

In our applications, we will take $\mathscr{P}=\cP^{(\mathtt{d})}$, $\cA=E_\mathtt{d}(\cB_{\mathtt{d}}^{(i)})$, and $\cA_p=E_{\mathtt{d}p}(\cB_{\mathtt{d}}^{(i)}\times\cC_{p}^{(i)})$. Let $\cR^{(\mathtt{d})}$ be the subset of $\cR$ which are coprime to $\mathtt{d}$, then for all $d\in\cR^{(\mathtt{d})}$, we have
\[
  \cA_d=\bigcap_{p|d}E_{\mathtt{d}p}(\cB_{\mathtt{d}}^{(i)}\times\cC_{p}^{(i)})=E_{\mathtt{d}d}(\cB_{\mathtt{d}}^{(i)}\times\cC_{d}^{(i)})
\]
The sizes of $\cA_d$ will be estimated using versions of the effective Chebotarev theorem, which we now recall.

\subsubsection{Background on effective Chebotarev}
For any number field $k$, let $n_k=[k:\Q]$ and $\Delta_k$ be its absolute discriminant. Given a finite extension $l/k$, let $\mathcal{D}(l/k)$ denote the set of rational primes that lie below the ramified finite places of $l/k$. Define
\[
  M(l/k)=\abs{G}\Delta_k^{\frac{1}{n_k}}\prod_{p\in\mathcal{D}(l/k)} p.
\]
Suppose that $l/k$ is Galois with Galois group $G$. Let $C\subseteq G$ be a union of conjugacy classes. For any positive real number $X$, define
\[
  \newcommand{\p}{\mathfrak{p}}
  \pi(X,C,l/k):=|\{\p\,|\,\p\text{ is unramified in }l,\ \Frob_\p\in C,\ \mathbf{N}_{k/\Q}\p\leq X\}|.
\]
We are interested in the error term
\[
    R(X,C,l/k):=\pi(X,C,l/k)-\frac{\abs{C}}{\abs{G}}\mathrm{Li}(X).
\]
The following result is equation ($20_\mathrm{R}$) in \cite[\S 2.4]{S1}, which is an improvement to Lagarias--Odlyzko's original result \cite{LO}.
\begin{thm}\label{thm:Ch}
    Assume the GRH holds for the Artin $L$-functions of irreducible representations of $l/k$, then
    \[
        \abs{R(X,C,l/k)}\ll\abs{C}n_k X^{\frac{1}{2}}\big(\log X+\log M(l/k)\big).
    \]
\end{thm}

Assuming the AHC, Murty--Murty--Saradha obtained the following improvement \cite[Corollary 3.10]{MMS}.
\begin{thm}\label{thm:Res}
    Let $H$ be a subgroup of $G$ such that the GRH and AHC hold for the Artin $L$-functions of the irreducible characters of $H$. Suppose $H$ meets every conjugacy class contained in $C$, then
    \[
        R(X,C,l/k)\ll\abs{C}^{\frac{1}{2}}[G:H]^{\frac{1}{2}} n_k X^{\frac{1}{2}}\big(\log X+\log M(l/k)\big).
    \]
\end{thm}

If $\abs{C}$ is of roughly the same size as $\abs{G}$, then this estimate saves a factor of $\abs{H}^{\frac{1}{2}}$ compared to Theorem~\ref{thm:Ch}. We therefore want to take $\abs{H}$ as large as possible.

\begin{remark}
    This theorem is a corollary of a precise result \cite[Proposition 3.9]{MMS}. While it suffices for our purpose, we expect that a finer analysis of centralizers in finite reductive groups can give a better exponent.
\end{remark}

\subsubsection{Applications}\label{ss:Estimates}
We now apply the above results to the various terms appearing in Theorem~\ref{thm:SelbergSieve}. Fix $i\in\cI$. For each $d\in\cR$, define
\[
  \alpha_d^{(i)}=\frac{|\cB_d^{(i)}|}{\abs{H_d}},\quad \beta_d^{(i)}=\frac{|\cC_d^{(i)}|}{\abs{H_d}}.
\]
Then we have the relations
\[
    \alpha_d^{(i)}=\prod_{p|d}\alpha_p^{(i)},\quad\beta_d^{(i)}=\prod_{p|d}\beta_p^{(i)},\quad \alpha_p^{(i)}+\beta_p^{(i)}=1
\]
for all $p\in\cP$ and $d\in\cR$. If $d_1,d_2$ are coprime numbers in $\cR$, then in the above notations,
\begin{equation}\label{eq:E}
    \big|E_{d_1d_2}(\cB_{d_1}^{(i)}\times \cC_{d_2}^{(i)})\big|=\frac{1}{[G_{d_1d_2}':H_{d_1d_2}]}\alpha_{d_1}^{(i)}\beta_{d_2}^{(i)}\mathrm{Li}(X)+R(X,\cB_{d_1}^{(i)}\times \cC_{d_2}^{(i)},K_{d_1d_2}/K).
\end{equation}
Using Theorem~\ref{thm:Ch} alone, this error term is of the order $|G_{d_1d_2}|X^{\frac{1}{2}+\varepsilon}$. 

In this subsection, we show that the error term for $\big|E_d(\cB_d^{(i)})\big|$ can be improved using Theorem~\ref{thm:Res}. This yields the required estimate for \eqref{eq:E} using the principle of inclusion-exclusion. The main result is Proposition~\ref{prop:E_d}.

For each prime $p$, fix a Borel subgroup $\mb{B}_p\subseteq\G_p^\mathrm{ss}$, and let $$B_p=\mb{B}_p(\F_p).$$ Define $B_d=\prod_{p|d}B_p$ and $B_d'=B_d\cap G_d'$. This is a subgroup of $G_d'$. The next two lemmas verify that it satisfies the conditions in Theorem~\ref{thm:Res}.

\begin{lm}
    All conjugacy classes in $\cB_d^{(i)}$ intersect $B_d'$. 
\end{lm}
\begin{proof}
    By the Borel conjugacy theorem \cite[Th\'eorème 4.13(b)]{BT}, this result holds in $G_d$. To conclude, we will show that the conjugating element can be chosen to lie in $I_d$. It is enough to do this for $d=p$, a prime.
    
    Suppose $b\in\cB_p'$ and $gbg^{-1}\in B_p'$, where $g\in G_p$. Let $\mb{T}_p\subseteq\mb{B}_p$ be a maximal torus. Then the commutative diagram appearing $\S\ref{ss:Isog}$ shows that $\mb{T}_p(\F_p)$ intersects all cosets of $I_p$. Choose $t\in\mb{T}_p(\F_p)$ so that $t$ and $g$ are in the same coset, then $t^{-1}g$ is an element of $I_p$ conjugating $b$ to $B_p'$.
\end{proof}

\begin{lm}\label{lem:Supersolvable}
    All irreducible representations of $B_d'$ are induced from abelian characters. Consequently, the AHC holds for $K_d/K_d^{B_d'}$, and the GRH holds provided it holds for all Hecke $L$-functions.
\end{lm}
\begin{proof}
    We will show that if $p\in\cP$, then $B_p$ is supersolvable. The group $B_d'$ is a subquotient of their product, so it is also supersolvable. The first claim of the lemma is an elementary group theory fact \cite[Theorem 12.8.5]{Sc}. The second part of the lemma is a standard consequence of class field theory.

    Let $p\in\cP$. We first observe that $\G_p^\mathrm{ss}$ is actually split over $\F_p$. Indeed, by a result of Serre (cf.\ \cite[Corollary 3.8]{Ch}), the element $\rho_p(\Frob_v)\in\G_p(\Q_p)$ generates a maximal torus for a density one set of places $v$. On the other hand, $p$ splits completely in $M$, so this element lies in a split torus. Therefore, we have constructed a split maximal torus of $\G_p$ over $\Q_p$, which also implies $\G_p^\mathrm{ss}$ is split.
    
    Fix a faithful representation $\G_p^\mathrm{ss}\hookrightarrow\GL_N$ defined over $\F_p$, we see that $B_p$ is isomorphic to a subgroup of $R_N$, the group of upper triangular matrices in $\GL_N(\F_p)$. Write the derived series of $R_N$ as
    \[
        R_N\triangleright S_N\triangleright S_{N-1}\triangleright\cdots\triangleright S_0=\{1\}.
    \]
    Then a standard calculation shows that
    \[
        S_k=\{g\in R_N\,|\,g_{ii}=1,\ g_{ij}=0\text{ if }0<j-i\leq N-k\}.
    \]
    For $v=1,\cdots,k-1$, let
    \[
        S_{k,v}=\{g\in S_k\,|\,g_{i,i+N-k}=0\text{ if }i\leq v\}. 
    \]
    Then we have a normal series $S_k\triangleright S_{k,1}\triangleright\cdots\triangleright S_{k,k-1}=S_{k-1}$ where each quotient is isomorphic to $\F_p$. Conjugation by a diagonal matrix preserves $S_{k,v}$, so all of the groups are normal in $R_N$. This normal series shows that $R_N$ is a supersolvable group, which completes the proof.
\end{proof}

\begin{prop}\label{prop:B_d}
  For each prime $p$, let
  \begin{equation}\label{eqn:Np}
    N_p=\abs{\G_p^\mathrm{ss}(\F_p)}\cdot\abs{\mb{B}_p(\F_p)}^{-\frac{1}{2}}.
  \end{equation}
  Assume the GRH for Hecke $L$-functions. Then for all $d\in\cR$, we have
  \[
    \big|E_d(\cB_d^{(i)})\big|=\frac{1}{[G_d':H_d]}\alpha_{d}^{(i)}\mathrm{Li}(X)+O_{A,K}\Big(\prod_{p|d}N_p\cdot X^{\frac{1}{2}}(\log X+\log d)\Big). 
  \]
\end{prop}
\begin{proof}
Applying Theorem~\ref{thm:Res} to the extension $l/k=K_d/K$ with the subgroup $H=B_d'$ gives the error term
\begin{align*}
    \abs{R(X,\cB_d^{(i)},K_d/K)}&\ll (|\cB_d^{(i)}|\cdot[G_d':B_d'])^{\frac{1}{2}} n_K X^{\frac{1}{2}}\big(\log X+\log M(K_d/K)\big)\\
    &\leq (\abs{G_d}\cdot[G_d:B_d])^{\frac{1}{2}} n_K X^{\frac{1}{2}}\big(\log X+\log M(K_d/K)\big)\\
    &=\prod_{p|d}N_p\cdot n_K X^{\frac{1}{2}}\big(\log X+\log M(K_d/K)\big). 
\end{align*}
For the terms beside $N_p$, recall that $K_d/K$ is unramified away from the places of $K$ dividing $d\Delta_A$, so
\begin{align*}
    \log M(K_d/K)&=\log\abs{G_d'}+\frac{\log(\Delta_K)}{n_K}+\sum_{p|d}\log p+\sum_{p|\Delta_A}\log p\\
    &\leq\sum_{p|d}\log\abs{\G_p^\mathrm{ss}(\F_p)}+\log d+O_{A,K}(1)\\
    &\ll_{A,K}\log d
\end{align*}
where for the final estimate, we used the trivial bound $\abs{\G_p^\mathrm{ss}(\F_p)}\leq\abs{\mathrm{GSp}_{2g}(\F_p)}$. 
\end{proof}

We now apply this to calculate the intersection terms in the Selberg sieve.

\begin{prop}\label{prop:E_d}
    Assume the GRH for Hecke $L$-functions. If $d\in\cR$ is coprime to $\mathtt{d}$, then
    \[
        \bigg|\bigcap_{p|d}E_{\mathtt{d}p}(\cB_{\mathtt{d}}^{(i)}\times\cC_{p}^{(i)})\bigg|=\frac{1}{\abs{\cI}}\alpha_{\mathtt{d}}^{(i)}\beta_d^{(i)}\mathrm{Li}(X)+O_{A,K}\Big(2^{\omega(d)}\prod_{p|\mathtt{d}d}N_p\cdot X^{\frac{1}{2}}\big(\log X+\log(\mathtt{d}d)\big)\Big).
    \]
\end{prop}
\begin{proof}
    We have the equality
    \[
      \bigcap_{p|d}E_{\mathtt{d}p}(\cB_{\mathtt{d}}^{(i)}\times\cC_{p}^{(i)})=E_{\mathtt{d}}(\cB_{\mathtt{d}}^{(i)})-\bigcup_{p|d}E_{\mathtt{d}p}(\cB_{\mathtt{dp}}^{(i)}).
    \]
    To see this, observe that $\cB_p^{(i)}\sqcup\cC_p^{(i)}$ is the coset of $g_{i,p}H_p$, and the condition defining the set $E_\mathtt{d}(\cB_\mathtt{d}^{(i)})$ already forces the Frobenius element to be in this coset (cf. Remark~\ref{rem:d}).
    
    The principle of inclusion-exclusion now gives
\[
    \bigg|\bigcap_{p|d}E_{\mathtt{d}p}(\cB_{\mathtt{d}}^{(i)}\times\cC_{p}^{(i)})\bigg|=\sum_{d'|d}\mu(d')\big|E_{\mathtt{d}d'}(\cB_{\mathtt{d}d'}^{(i)})\big|.
\]
    Apply the previous proposition to the right hand side. There %is a total of 
    are $2^{\omega(d)}$ terms in the above sum, so the error term has the required form by the triangle inequality. In the main term, the coefficient in front of $\mathrm{Li}(X)$ is given by 
    \[
        \sum_{d'|d}\mu(d')\cdot\frac{1}{\abs{\cI}}\alpha_{\mathtt{d}d'}^{(i)}=\frac{1}{\abs{\cI}}\alpha_{\mathtt{d}}^{(i)}\prod_{p|d}(1-\alpha_p^{(i)})=\frac{1}{\abs{\cI}}\alpha_{\mathtt{d}}^{(i)}\beta_d^{(i)}
    \]
    where we used the observation that $[G_d':H_d]=[G_\infty':H_\infty]=\abs{\cI}$ for $\mathtt{d}|d$.
\end{proof}

\subsubsection{Point counting}\label{ss:Conclusion}
We will now estimate $N_p$, defined in~\eqref{eqn:Np}.

\begin{prop}\label{prop:PointCount}
    Let $\G$ be the Mumford--Tate group of $A_{/K}$ and $\G^\mathrm{ss}$ its semisimple quotient. Put
    \[
        \gamma=\frac{1}{4}(3\dim\G^\mathrm{ss}-\rank\G^\mathrm{ss}).
    \]
    There exists a constant $C$ depending only on $g$ such that $$N_p\leq Cp^\gamma$$ for all primes $p$.
\end{prop}

The key feature of this proposition is that neither $C$ nor $\gamma$ depend on the prime $p$, even though the monodromy groups $\G_p$ are not assumed to interpolate into a group over $\Q$. The proposition follows by combining Lemma~\ref{lem:WeilBound} and Corollary~\ref{cor:Comparision} below.

\begin{lm}\label{lem:WeilBound}
    There exists a constant $C$ depending only on $g$ such that for all prime $p$,
    \[
        N_p\leq Cp^{\dim\G_p^\mathrm{ss}-\frac{1}{2}\dim\mb{B}_p}.
    \]
\end{lm}
\begin{proof}
In view of Lemma~\ref{lm:IsogCount} we may replace $\G_p^\mathrm{ss}$ with $\G_p^\mathrm{sc}$ in the definition of $N_p$. The base change $(\G_{p}^\mathrm{sc})_{/\overline{\F}_p}$ is a simply connected semisimple group, so it is a product of simply connected split simple groups. These are classified by Dynkin diagrams. Since the rank is bounded (say by $2g$), there are only a finite number of possibilities. Moreover, each of them is obtained from a Chevalley group over $\Z$ by base change.

Now, $\G_p^\mathrm{sc}(\F_p)$ is the set of fixed points of the Frobenius operator $F$ acting on $\G_p^\mathrm{sc}(\overline{\F}_p)$. The Grothendieck--Lefschetz trace formula gives
\[
  \abs{\G_p^\mathrm{sc}(\F_p)}=\sum_{i=0}^{2\dim\G_p^\mathrm{sc}}(-1)^i\mathrm{Tr}\big(F,\mathrm{H}^i_c\big((\G_p^\mathrm{sc})_{/\overline{\F}_p},\mathrm{\Q}_\ell\big)\big)
\]
The top degree term is $p^{\dim\G_p^\mathrm{sc}}$ since $\G_p^\mathrm{sc}$ is geometrically connected. Deligne's Weil bound gives an upper bound for the Frobenius eigenvalues on each of the remaining terms. In particular, it implies
\[
  \left|p^{-\dim\G_p^\mathrm{sc}}\abs{\G_p^\mathrm{sc}(\F_p)}-1\right|\leq b p^{-\frac{1}{2}}
\]
where $b$ is the sum of the dimensions of all lower degree cohomology groups. The argument in the previous paragraph gives a finite list of possibilities for $b$ independently of $p$. Therefore, we get positive constants $c_1(g),c_2(g)$ depending only on $g$ such that
\[
  c_1p^{\dim\G_p^{\mathrm{sc}}}\leq\abs{\G_p^\mathrm{sc}(\F_p)}\leq c_2p^{\dim\G_p^{\mathrm{sc}}}
\]
for all primes $p$.

By applying the exact same argument to the term $\abs{\mb{B}_p(\F_p)}$, the assertion follows.
\end{proof}

Let $\G$ be the Mumford--Tate group of $A$. For our purpose, it suffices to know that this is a reductive group over $\Q$ which ``contains'' all $p$-adic monodromy groups, made precise by the following theorem of Deligne \cite{De}.
\begin{thm}\label{thm:De}
  For all primes $p$, we have\footnote{Recall that we are assuming all monodromy groups are connected.} $\G_{p/\Q_p}\subseteq\G\times_\Q\Q_p$.
\end{thm}

\begin{cor}\label{cor:Comparision}
  Let $\G^\mathrm{ss}$ denote the semisimple quotient of $\G$, then for all primes $p$,
  \[
    \dim\G_p^\mathrm{ss}-\frac{1}{2}\dim\mb{B}_p\leq\frac{1}{4}(3\dim\G^\mathrm{ss}-\rank\G^\mathrm{ss})
  \]
\end{cor}
\begin{proof}
  Fix a prime $p$. For simplicity, we base change the pertinent  groups to $\overline{\Q}_p$. This does not change its dimension or rank.

  Recall that $\mb{B}_p$ is a Borel subgroup of $\G_p^\mathrm{ss}$. Let $\mathbf{N}_p$ be its unipotent radical. We have the relations
  \[
    \dim\G_p^\mathrm{ss}=2\dim\mb{B}_p-\rank\G_p^\mathrm{ss}=\dim\mb{B}_p+\dim\mathbf{N}_p
  \]
  Rearranging this, we get
  \[
    \dim\G_p^\mathrm{ss}-\frac{1}{2}\dim\mb{B}_p=\frac{1}{4}(3\dim\G^\mathrm{ss}_p-\rank\G_p^\mathrm{ss})=\frac{1}{2}\dim\G_p^\mathrm{ss}+\frac{1}{2}\dim\mathbf{N}_p
  \]
  It remains to show that $\dim\G_p^\mathrm{ss}\leq\dim\G^\mathrm{ss}$ and $\dim\mathbf{N}_p\leq\dim\mathbf{N}$.
  
  The previous theorem gives an embedding $\G_p\subseteq\G$ of reductive groups, and hence the inclusion of the derived subgroups $\G_p^\der\subseteq\G^\der$. Moreover, the derived subgroup is isogenous to the semisimple quotient, so they have the same dimension and rank. The first inequality follows. For the second part, the image of $\mathbf{N}_p$ is again a unipotent subgroup, so it lies in a conjugate of $\mathbf{N}$.
\end{proof}

\subsubsection{Conclusion of the proof}
We now assemble the pieces to conclude the proof.
\begin{lm}\label{lem:1}
  If $v$ is a finite place of $K$ not dividing $\mathfrak{N}$, then
  \[
    \log|\Delta_{F(A,v)}|\ll_g\log\mathbf{N}_{K/\Q}v.
  \]
\end{lm}
\begin{proof}
  Let $P(T)\in\Z[T]$ be the characteristic polynomial of $\rho_p(\Frob_v)$. Let $\alpha_1,\cdots,\alpha_{2g}$ be the roots of $P(T)$, then $F(A,v)$ is the compostium of the fields $\Q(\alpha_i)$. Therefore,
  \[
    \frac{\log|\Delta_{F(A,v)}|}{[F(A,v):\Q]}\leq\sum_{i=1}^{2g}\frac{\log|\Delta_{\Q(\alpha_i)}|}{[\Q(\alpha_i):\Q]}.
  \]
  Let $q=\mathbf{N}_{K/\Q}v$, then all the roots of $P(T)$ have complex absolute value $q^{\frac{1}{2}}$, so
  \[
    \log\abs{\Delta_{\Q(\alpha)}}\leq\log\abs{\disc P(T)}=\sum_{i\neq j}\log\abs{\alpha_i-\alpha_j}\leq ((2g)^2-2g)\log\big(2q^{\frac{1}{2}}\big).
  \]
  Combining the above two inequalities gives the claim.
\end{proof}

\begin{thm}\label{thm:Main}
    Assume the GRH. For any $\varepsilon>0$, we have
    \[
        \abs{S_{A,M}(X)}\ll_{A,K,\varepsilon} X^{1-\frac{1}{4\gamma+2}+\varepsilon},
    \]
    where $\gamma=\frac{1}{4}(3\dim\G^\mathrm{ss}-\rank\G^\mathrm{ss})$ from Proposition~\ref{prop:PointCount}.
\end{thm}
\begin{proof}
Fix $\varepsilon>0$, then Proposition~\ref{prop:Sieve} gives an equality
\[
  \abs{S_{A,M}(X)}=\sum_{i\in\cI}\bigg|E_\mathtt{d}(\cB_{\mathtt{d}}^{(i)})-\bigcup_{p\in\cP^{(\mathtt{d})}}E_{\mathtt{d}p}(\cB_{\mathtt{d}}^{(i)}\times\cC_{p}^{(i)})\bigg|.
\]
Fix an index $i\in\cI$. We will apply Theorem~\ref{thm:SelbergSieve} to the corresponding term in the above sum. Take $\cA=E_\infty(g_i H_\infty)$, sieving primes $\cP^{(\mathtt{d})}$, and excluded subsets $\cA_p=E_{\mathtt{d}p}(\cB_{\mathtt{d}}^{(i)}\times\cC_{p}^{(i)})$. By combining Propositions~\ref{prop:E_d} and ~\ref{prop:PointCount}, we get the estimate
\begin{align*}
    |\cA_d|=\bigg|\bigcap_{p|d}E_{\mathtt{d}p}(\cB_{\mathtt{d}}^{(i)}\times\cC_{p}^{(i)})\bigg|&=\frac{1}{\abs{\cI}}\alpha_{\mathtt{d}}^{(i)}\beta_d^{(i)}\mathrm{Li}(X)+O_{A,K}\big((2C)^{\omega(d)}d^{\gamma}X^{\frac{1}{2}}(\log X+\log d)\big)\notag\\
    &=\frac{1}{\abs{\cI}}\alpha_{\mathtt{d}}^{(i)}\beta_d^{(i)}\mathrm{Li}(X)+O_{A,K,\varepsilon}\big(d^{\gamma+\varepsilon}X^{\frac{1}{2}}(\log X+\log d)\big).
\end{align*}
By definition, the function $d\mapsto\beta_d^{(i)}$ is multiplicative. Let $p\in\cP^{(\mathtt{d})}$, then
\[
  \beta_p^{(i)}=\frac{|\cC_p\cap g_{i,p}H_p|}{\abs{H_p}}=1-\frac{|\cB_p\cap g_{i,p}H_p|}{\abs{H_p}}.
\]
Recall that $H_p$ contains $I_p$, so each $g_{i,p}H_p$ is a union of cosets of $I_p$. Over each such coset, Proposition~\ref{cor:CosetUpperBound} gives uniform upper and lower bounds for its intersection with $\cB_p$. It follows that $\beta_p^{(i)}$ is uniformly bounded away from 0 and 1, and the bounds only depend on the dimension $g$.

Having verified all of the hypotheses, Theorem~\ref{thm:SelbergSieve} gives the estimate
\begin{multline*}
  \bigg|E_\mathtt{d}(\cB_{\mathtt{d}}^{(i)})-\bigcup_{p\in\cP^{(\mathtt{d})}}E_{\mathtt{d}p}(\cB_{\mathtt{d}}^{(i)}\times\cC_{p}^{(i)})\bigg|\\
  \ll_{A,K,\varepsilon}\frac{1}{\abs{\cI}}\alpha_{\mathtt{d}}^{(i)}\cdot \frac{\mathrm{Li}(X)}{\pi_{\cP^{(\mathtt{d})}}(z)}+\left(\frac{z^{1+\varepsilon}}{\pi_{\cP^{(\mathtt{d})}}(z)}\right)^2 X^{\frac{1}{2}}\sum_{\substack{d_1,d_2\leq z\\ d_1,d_2\in\cR^{(\mathtt{d})}}}\frac{[d_1,d_2]^{\gamma+\varepsilon}}{d_1d_2}(\log X+\log[d_1,d_2]).
\end{multline*}
In the sum for the second term, use the trivial bound $[d_1,d_2]\leq d_1d_2$ and extend the sum to all pairs $d_1,d_2\leq z$. In particular, $\log[d_1,d_2]\leq\log(z^2)$. Therefore, 
\begin{align*}
    \sum_{\substack{d_1,d_2\leq z\\ d_1,d_2\in\cR^{(\mathtt{d})}}}\frac{[d_1,d_2]^{\gamma+\varepsilon}}{d_1d_2}(\log X+\log[d_1,d_2])&\leq \sum_{d_1,d_2\leq z}(d_1d_2)^{\gamma-1+\varepsilon}(\log X+2\log z)\\
    &=(\log X+2\log z)\Big(\sum_{d\leq z} d^{\gamma-1+\varepsilon}\Big)^2\\
    &\ll_\gamma z^{2\gamma+2\varepsilon}(\log X+\log z).
\end{align*}
which gives
\begin{equation}\label{eqn:1}
  \bigg|E_\mathtt{d}(\cB_{\mathtt{d}}^{(i)})-\bigcup_{p\in\cP^{(\mathtt{d})}}E_{\mathtt{d}p}(\cB_{\mathtt{d}}^{(i)}\times\cC_{p}^{(i)})\bigg|\ll_{A,K,\varepsilon}\frac{X}{\pi_{\cP^{(\mathtt{d})}}(z)}+\left(\frac{z^{1+\varepsilon}}{\pi_{\cP^{(\mathtt{d})}}(z)}\right)^2X^{\frac{1}{2}}z^{2\gamma+2\varepsilon}\log(Xz).
\end{equation}

It remains to bound $\pi_{\cP^{(\mathtt{d})}}(z)$ from below. By definition,
\[
  \pi_{\cP^{(\mathtt{d})}}(z)=\pi(z,\{1\},M/\Q)-O_{A,K}(1)
\]
By the effective Chebotarev theorem (Theorem~\ref{thm:Ch}, also cf.~\cite[Equation (14\textsubscript{R})]{S1}), we have
\[
  \pi(z,\{1\},M/\Q)=\frac{1}{[M:\Q]}\mathrm{Li}(z)+O_g\big(z^{\frac{1}{2}}(\log z+\log|\Delta_M|)\big).
\]
We may assume $\log|\Delta_M|\ll_g\log X$, since otherwise $S_{A,M}(X)=\emptyset$ by Lemma~\ref{lem:1}. Now choose $z=X^\beta$, where $\beta=\frac{1}{4\gamma+2}$, then
\[
  \pi_{\cP^{(\mathtt{d})}}(z)=\frac{1}{[M:\Q]}\mathrm{Li}(X^\beta)+O_{g,\beta}(X^{\frac{\beta}{2}}\log X)-O_{A,K}(1)\gg_{A,K,\varepsilon} X^{\beta-\varepsilon}.
\]
Substituting this in equation~\ref{eqn:1} and summing over all $i\in\cI$ gives the desired result.
\end{proof}

\begin{remark}
  We only needed the error term in the effective Chebotarev theorem on average. The situation is analogous to the Bombieri--Vinogradov theorem, except that it deals with the family of abelian extensions $\{\Q(\zeta_q)/\Q\}$, and we have a family of non-abelian Lie-extensions $\{K_d/K\}$. There are some works when the Galois group is fixed (cf.~\cite{PTW}), but we are not aware of any work in our setting.
\end{remark}

\bibliographystyle{plain}
\bibliography{Frob_ref}

\end{document}

%% file: amssymbol.tex
\usepackage{amsmath}
\usepackage{amscd,amsthm,amssymb,amsfonts}
\usepackage{mathrsfs}
\usepackage{dsfont}
\usepackage{stmaryrd}
\usepackage{euscript}
\usepackage{expdlist}

\usepackage{enumitem}
\input xy
\xyoption{all}
\theoremstyle{plain}
\newtheorem{thm}{Theorem}[section]

\newtheorem*{thm*}{Theorem}

\newtheorem{lm}[thm]{Lemma}
\newtheorem{cor}[thm]{Corollary}
\newtheorem*{cor*}{Corollary}

\newtheorem{prop}[thm]{Proposition}
\newtheorem*{conj*}{Conjecture}

\theoremstyle{remark}
\newtheorem{remark}[thm]{Remark}
\newtheorem*{thank}{Acknowledgments}

\theoremstyle{definition}
\newtheorem*{defn*}{Definition}

\newtheorem{defn}[thm]{Definition}

\newtheorem{hypothesis}[thm]{Hypothesis}

\newcommand{\nc}{\newcommand}

\newcommand{\beq}{\begin{equation}}
\newcommand{\eeq}{\end{equation}}
\newcommand{\bpmx}{\begin{pmatrix}}
\newcommand{\epmx}{\end{pmatrix}}
\newcommand{\bbmx}{\begin{bmatrix}}
\newcommand{\ebmx}{\end{bmatrix}}
\newcommand{\wh}{\widehat}

\newcommand{\beqcd}[1]{\begin{equation*}\label{#1}\tag{#1}}
\newcommand{\eeqcd}{\end{equation*}}

\numberwithin{equation}{section}

\def\makeop#1{\expandafter\def\csname#1\endcsname
  {\mathop{\rm #1}\nolimits}\ignorespaces}
\makeop{Hom}   \makeop{End}   \makeop{Aut}   %\makeop{Isom}
\makeop{Pic} \makeop{Gal}       \makeop{Div} \makeop{Lie}
\makeop{PGL}   \makeop{Corr} \makeop{PSL} \makeop{sgn} \makeop{Spf}
 \makeop{Tr} \makeop{Nr} \makeop{Fr} \makeop{disc}
\makeop{Proj} \makeop{supp} \makeop{ker}   \makeop{Im} \makeop{dom}
\makeop{coker} \makeop{Stab} \makeop{SO} \makeop{SL} \makeop{SL}
\makeop{Cl}    \makeop{cond} \makeop{Br} \makeop{inv} \makeop{rank}
\makeop{id}    \makeop{Fil} \makeop{Frac}  \makeop{GL} \makeop{SU}
\makeop{Trd}   \makeop{Sp} \makeop{Tr}    \makeop{Trd} \makeop{Res}
\makeop{ind} \makeop{depth} \makeop{Tr} \makeop{st} \makeop{Ad}
\makeop{Int} \makeop{tr}    \makeop{Sym} \makeop{can} \makeop{SO}
\makeop{torsion} \makeop{GSp} \makeop{Tor}\makeop{Ker} \makeop{rec}
\makeop{Ind} \makeop{Coker}
 \makeop{vol} \makeop{Ext} \makeop{gr} \makeop{ad}
 \makeop{Gr}\makeop{corank} \makeop{Ann}
\makeop{Hol} %Holomorphic
\makeop{Fitt} \makeop{Mp} \makeop{CAP}
\DeclareFontFamily{U}{wncy}{}
\DeclareFontShape{U}{wncy}{m}{n}{<->wncyr10}{}
\DeclareSymbolFont{mcy}{U}{wncy}{m}{n}
\DeclareMathSymbol{\Sha}{\mathord}{mcy}{"58}

\def\makebb#1{\expandafter\def
  \csname bb#1\endcsname{{\mathbb{#1}}}\ignorespaces}
\def\makebf#1{\expandafter\def\csname bf#1\endcsname{{\bf
      #1}}\ignorespaces}
\def\makegr#1{\expandafter\def
  \csname gr#1\endcsname{{\mathfrak{#1}}}\ignorespaces}
\def\makescr#1{\expandafter\def
  \csname scr#1\endcsname{{\EuScript{#1}}}\ignorespaces}
\def\makecal#1{\expandafter\def\csname cal#1\endcsname{{\mathcal
      #1}}\ignorespaces}
\def\doLetters#1{#1A #1B #1C #1D #1E #1F #1G #1H #1I #1J #1K #1L #1M
                 #1N #1O #1P #1Q #1R #1S #1T #1U #1V #1W #1X #1Y #1Z}
\def\doletters#1{#1a #1b #1c #1d #1e #1f #1g #1h #1i #1j #1k #1l #1m
                 #1n #1o #1p #1q #1r #1s #1t #1u #1v #1w #1x #1y #1z}
\doLetters\makebb   \doLetters\makecal  \doLetters\makebf
\doLetters\makescr
\doletters\makebf   \doLetters\makegr   \doletters\makegr

\normalsize

\def\der{{\rm der}}
\makeop{Ram} \makeop{Rep} \makeop{mass}

\makeop{Bl}
\def\abs#1{\left|#1\right|}

\def\cA{{\mathcal A}}  %automorphic forms
\def\cB{\EuScript B}

\def\cF{{\mathcal F}}  %Hida family

\def\cI{\mathcal I}

\def\cR{{\mathcal R}}

\def\cP{{\mathcal P}}
\def\cC{\mathcal C}

\newcommand{\Z}{\mathbb Z}
\newcommand{\Q}{\mathbb Q}
\newcommand{\R}{\mathbb R}

\newcommand{\F}{\mathbb F}

  \nc{\opp}{\mathrm{opp}} \nc{\ul}{\underline}

\def\Sg{{\varSigma}}  %%CM type

\def\ndivide{\nmid}